\documentclass[12pt]{amsart}
\usepackage{amsmath, amsthm, amssymb, tikz}
\usepackage[initials]{amsrefs}

\title{Littlestone and VC-dimension of families of zero sets}
\author[V Guingona]{Vincent Guingona}
\author[A Kolesnikov]{Alexei Kolesnikov}
\author[J Nierwinski]{Julie Nierwinski}
\author[R Soucy]{Richard Soucy}
\address{Towson University, 7800 York Rd., Towson, MD, 21252}
\email{vguingona@towson.edu}
\email{akolesnikov@towson.edu}
\date{\today}
\thanks{This work was partially supported by National Security Agency (NSA) grant H98230-20-1-0005, which funds a Research Experiences for Undergraduates program at Towson University. This work also received financial support from the Towson University Jess and Mildred Fisher College of Science and Mathematics.}

\newtheorem{theorem}{Theorem}[section]
\newtheorem{corollary}[theorem]{Corollary}
\newtheorem{lemma}[theorem]{Lemma}
\newtheorem{proposition}[theorem]{Proposition}

\theoremstyle{remark}
\newtheorem{remark}[theorem]{Remark}
\newtheorem{example}[theorem]{Example}

\theoremstyle{definition}
\newtheorem{definition}[theorem]{Definition}

\newcommand{\Pow}{\mathcal{P} }
\newcommand{\CC}{\mathcal{C} }
\newcommand{\NN}{\mathbb{N} }
\newcommand{\RR}{\mathbb{R} }
\newcommand{\ZZp}{\mathbb{Z}^+ }
\newcommand{\ZO}{Z}
\newcommand{\oa}{\overline{a} }
\newcommand{\ob}{\overline{b} }
\newcommand{\oc}{\overline{c} }
\newcommand{\oee}{\overline{e} }
\newcommand{\of}{\overline{f} }
\newcommand{\og}{\overline{g} }
\newcommand{\ox}{\overline{x} }
\newcommand{\ozero}{\overline{0} }
\DeclareMathOperator{\ldim}{Ldim}
\DeclareMathOperator{\vcdim}{VCdim}
\DeclareMathOperator{\lden}{Lden}
\DeclareMathOperator{\vcden}{VCden}
\DeclareMathOperator{\Span}{span}
\DeclareMathOperator{\supp}{supp}

\begin{document}

\maketitle

\begin{abstract}
    We prove that, for any $d$ linearly independent functions from some set into a $d$-dimensional vector space over any field, the family of zero sets of all non-trivial linear combination of these functions has VC-dimension and Littlestone dimension $d-1$.  Additionally, we characterize when such families are maximal of VC-dimension $d-1$ and give a sufficient condition for when they are maximal of Littlestone dimension $d-1$.
\end{abstract}

\section{Introduction}
Complexity, in the sense of machine learning theory, of the sets of positivity of linear combinations of real-valued functions is fairly well understood. If $f:\RR^k\to \RR$ is a function, the \emph{set of positivity of $f$} is the set $\{\ox\in \RR^k : f(\ox)>0\}$. If $\{f_1,\dots,f_d\}$ is a set of linearly independent functions from $\mathbb R^k$ to $\mathbb R$, then the set system of sets of positivity of all the linear combinations of $\{f_1,\dots,f_d\}$ has the VC dimension $d$ (see \cite{cover1965} or \cite[Theorem A]{dudley1979}). If $X$ is a finite subset of the domain of the functions, by Sauer--Shelah lemma, the size of the induced set system on $X$ is bounded above by $\sum_{i=0}^d \binom{|X|}{i}$ (we abbreviate this sum as $\binom{|X|}{\le d}$). 

There are known sufficient conditions on functions guaranteeing that for suitable finite sets $X$, the induced set systems have maximal size $\binom{|X|}{\le d}$. Floyd shows in \cite{floyd1989} that a condition on linear dimension of the family of functions relative to the set $X$ and a condition on the number of zeros are sufficient for maximality. Johnson provides analytic sufficient conditions in \cite{johnson2014} for maximality of the set system. It is worth pointing out that set systems given by the sets of positivity generally have infinite Littlestone dimension.

We study set systems of zero-sets of non-trivial linear combinations of functions from a set to a field. One expects to see a lower combinatorial complexity of this set system, and indeed we show that both Littlestone dimension and VC-dimension of zero-sets of linear combinations of $d$ linearly independent functions is $d-1$. Somewhat surprisingly, the system of zero-sets can be maximal (we formally define the notion in Definition~\ref{def_maximal}). We obtain a characterization of maximality of the set system.

Bhaskar has studied a shatter function related to Littlestone dimension (called ``thicket dimension'' in \cite{Bhaskar}). Bhaskar showed that the size of the set system that well-labels the leaves of the tree of depth $n$ (this is a natural notion for the shatter function for Littlestone dimension) is bounded above by $\binom{n}{\le d}$, where $d$ is the Littlestone dimension. It is natural to ask for examples of maximal families for this shatter function.

In this paper, we establish the following results: Let $\{ f_1, \dots, f_d \}$ be a set of functions from a set $X$ to a vector space $F^d$, where $F$ is a field and $d$ is a positive integer, and let $\CC$ be the set of all zero-sets of non-trivial linear combinations of $f_1, \dots, f_d$.
\begin{description}
    \item [Theorem \ref{Theorem_VCLDim}] If $\{ f_1, \dots, f_d \}$ is linearly independent (in the vector space of all functions from $X$ to $F$), then $\CC$ has VC-dimension and Littlestone dimension $d-1$.
    \item [Theorem \ref{Theorem_MaxVCDim}] $\CC$ is maximum of VC-dimension $d-1$ if and only if the image of $(f_1, ..., f_d)$ is not contained in the union of finitely many proper subspaces of $F^d$.
    \item [Corollary \ref{Corollary_MaxLDim}] If the image of $(f_1, ..., f_d)$ is not contained in the union of finitely many proper subspaces of $F^d$, then $\CC$ is maximum of Littlestone dimension $d-1$.
\end{description}

\section{Preliminaries}

For this paper, let $\NN$ denote the set of natural numbers, which for us is the set of all non-negative integers (including $0$).  Let $\ZZp$ denote the set of positive integers (i.e., $\ZZp = \NN \setminus \{ 0 \}$).

For $n, k \in \NN$ with $k \le n$, let
\[
    \binom{n}{k} = \frac{n!}{k!(n-k)!} \text{ and } \binom{n}{\le k} = \sum_{i=0}^k \binom{n}{i}.
\]
For all sets $X$ and $Y$, let ${}^X Y$ denote the set of all functions from $X$ to $Y$.  For any set $X$ and $f : X \rightarrow \RR$, let $\supp(f) = \{ a \in X : f(a) \neq 0 \}$.  For all sets $X$ and all $n \in \NN$, let $\binom{X}{n}$ denote the set of all subsets of $X$ of size $n$.  Similarly, let
\[
    \binom{X}{\le n} = \bigcup_{i \le n} \binom{X}{i} \text{ and } \binom{X}{< n} = \bigcup_{i < n} \binom{X}{i}.
\]
For all $n \in \NN$, let
\[
    [n] = \{ k \in \NN : 0 \le k < n \}.
\]
(Note that $[0] = \emptyset$.)  Then, we have that, for all $0 \le k \le n$,
\[
    \left| \binom{[n]}{k} \right| = \binom{n}{k} \text{ and } \left| \binom{[n]}{\le k} \right| = \binom{n}{\le k}.
\]
    
\begin{definition}
    Let $X$ be a set, $\CC \subseteq \Pow(X)$, and $d \in \NN$.  A \emph{$(X,\CC)$-labeled tree of depth $d$} is a function $T$ from $\bigcup_{k \le d} {}^{[k]} [2]$ to $X \cup \CC$ such that
    \begin{enumerate}
        \item for all $0 \le k < d$ and $\sigma \in {}^{[k]} [2]$, $T(\sigma) \in X$; and
        \item for all $\tau \in {}^{[d]} [2]$, $T(\tau) \in \CC$
    \end{enumerate}
    For $0 \le k \le d$ and $\sigma \in {}^{[k]} [2]$, we call $(\sigma, T(\sigma))$ a \emph{node} if $k < d$ and a \emph{leaf} if $k = d$.  A leaf $(\tau, T(\tau))$ is \emph{well-labeled} if, for all $0 \le k < d$,
    \[
        T(\tau |_{[k]}) \in T(\tau) \Longleftrightarrow \tau(k) = 1.
    \]
    We say that $T$ is \emph{well-labeled} if all leaves of $T$ are well-labeled.  We say that $T$ of depth $d$ is \emph{level-balanced} if, for all $k < d$ and for all $\sigma_0, \sigma_1 \in {}^{[k]} [2]$, $T(\sigma_0) = T(\sigma_1)$.
\end{definition}

\begin{definition}
    Let $X$ be a set and $\CC \subseteq \Pow(X)$.
    \begin{enumerate}
        \item The \emph{Littlestone dimension} of $\CC$, denoted $\ldim(\CC)$, is the largest natural number $d$ such that there exists a well-labeled $(X,\CC)$-labeled tree of depth $d$.  If there exists a well-labeled $(X,\CC)$-labeled tree of depth $d$ for all natural numbers $d$, we say $\CC$ has \emph{infinite Littlestone dimension}, denoted $\ldim(\CC) = \infty$.  If there exists no well-labeled $(X,\CC)$-labeled tree of depth $0$, we say $\ldim(\CC) = -\infty$.
        \item The \emph{Littlestone shatter function} of $\CC$, denoted $\rho_{\CC}$, is the function from $\NN$ to $\NN$ given by, for all $n \in \NN$, the maximum number of well-labeled leaves of a $(X, \CC)$-labeled tree $T$ of depth $n$.
        \item The \emph{Littlestone density} of $\CC$, denoted $\lden(\CC)$, is the infimum over all positive $\ell \in \RR$ such that there exists $K \in \RR$ such that, for all $n \ge 1$, $\rho_{\CC}(n) \le Kn^\ell$.
        \item If we restrict our attention to only level-balanced trees $T$ in the above definitions, we define instead the \emph{VC-dimension} of $\CC$, denoted $\vcdim(\CC)$, the \emph{VC-shatter function} of $\CC$, denoted $\pi_{\CC}$, and the \emph{VC-density} of $\CC$, denoted $\vcden(\CC)$, respectively.
    \end{enumerate}
\end{definition}

\begin{remark}\label{Remark_VCDimLDimBound}
    For any set $X$ and $\CC \subseteq \Pow(X)$, we clearly have
    \begin{enumerate}
        \item $\vcdim(\CC) \le \ldim(\CC)$,
        \item $\pi_{\CC}(n) \le \rho_{\CC}(n)$ for all $n \in \NN$, and
        \item $\vcden(\CC) \le \lden(\CC)$.
    \end{enumerate}
\end{remark}

The above definitions for the VC-dimension and the VC-shatter function are equivalent to the usual definitions given in terms of shattering.  In particular, Lemma \ref{Lemma_UsualVCDef} below can be found, for example, in \cite{chase2020}.

\begin{definition}
    Let $X$ be a set, $Y \subseteq X$, and $\CC \subseteq \Pow(X)$.  Let
    \[
        \CC |_Y = \{ A \cap Y : A \in \CC \},
    \]
    we call $\CC|_Y$ an \emph{induced set system on $Y$}.
    We say that $\CC$ \emph{shatters} $Y$ if $C |_Y = \Pow(Y)$.
\end{definition}

\begin{lemma}\label{Lemma_UsualVCDef}
    Let $X$ be a set and $\CC \subseteq \Pow(X)$.
    \begin{enumerate}
        \item the VC-dimension of $\CC$ is the largest $d \in \NN$ such that there exists $Y \in \binom{X}{d}$ with $| \CC |_Y | = 2^d$.
        \item the VC-shatter function of $\CC$ is given by, for all $n \in \NN$,
        \[
            \pi_{\CC}(n) = \max \left\{ \bigl| \CC |_Y \bigr| : Y \in \binom{X}{n} \right\}.
        \]
    \end{enumerate}
\end{lemma}

\begin{lemma}[Sauer--Shelah Lemma]
   For any set $X$ and $\CC \subseteq \Pow(X)$, if $\CC$ has VC-dimension $d$, then, for all $n \ge d$,
    \[
        \pi_\CC(n) \le \binom{n}{\le d}.
    \]
\end{lemma}

An analogous lemma holds in the Littlestone case.

\begin{lemma}[\cite{Bhaskar}]\label{Lemma_Bhaskar}
    For any set $X$ and $\CC \subseteq \Pow(X)$, if $\CC$ has Littlestone dimension $d$, then, for all $n \ge d$,
    \[
        \rho_\CC(n) \le \binom{n}{\le d}.
    \]
\end{lemma}

This motivates the following two definitions:

\begin{definition}
\label{def_maximal}
    Let $X$ be a set and $\CC \subseteq \Pow(X)$.
    \begin{enumerate}
        \item For $d \in \NN$, we say $\CC$ is \emph{maximal of VC-dimension $d$} if, for all $n \ge d$,
        \[
            \pi_\CC(n) = \binom{n}{\le d}.
        \]
        \item For $d \in \NN$, we say $\CC$ is \emph{maximal of Littlestone dimension $d$} if, for all $n \ge d$,
        \[
            \rho_\CC(n) = \binom{n}{\le d}.
        \]
    \end{enumerate}
\end{definition}

\begin{remark}\label{Remark_VCDenVCDim}
    For any set $X$ and $\CC \subseteq \Pow(X)$, by the Sauer-Shelah Lemma,
    \[
        \vcden(\CC) \le \vcdim(\CC).
    \]
    Similarly, by Lemma \ref{Lemma_Bhaskar},
    \[
        \lden(\CC) \le \ldim(\CC)
    \]
    Moreover, if $\CC$ is maximal of VC-dimension $d$, then
    \[
        \vcden(\CC) = \vcdim(\CC) = d
    \]
    and, if $\CC$ is maximal of Littlestone dimension $d$, then
    \[
        \lden(\CC) = \ldim(\CC) = d.
    \]
    If $\CC$ is maximal of VC-dimension $d$ and $\ldim(\CC) = d$, then $\CC$ is maximal of Littlestone dimension $d$ as well.  This is because, in this case, for all $n \ge d$,
    \[
        \binom{n}{\le d} = \pi_\CC(n) \le \rho_\CC(n) \le \binom{n}{\le d}.
    \]
\end{remark}

\section{VC-dimension and Littlestone dimension of zero sets.}

Let $X$ be a set and $F$ a field. The family of functions from $X$ to $F$ has a natural structure of a vector space over $F$ with the pointwise addition operation. The notion of linear independence below is taken in the sense of that vector space. In particular, if $X=F=\mathbb{F}_3$, then the set of functions $\{x,x^3\}$ is linearly dependent, because $x$ and $x^3$ define the same function from $\mathbb F_3$ to $\mathbb F_3$.

\begin{definition}
    Let $F$ be a field, $d \in \ZZp$, and $\oa, \ob \in F^d$.  Let $\oa \cdot \ob$ denote the usual dot product of $a$ and $b$.  That is,
    \[
        (a_0, \dots, a_{d-1}) \cdot (b_0, \dots, b_{d-1}) = \sum_{i=0}^{d-1} a_i b_i
    \]
    for all $a_0, \dots, a_{d-1}, b_0, \dots, b_{d-1} \in F$.  We say that $\oa, \ob \in F^d$ are \emph{orthogonal} if $\oa \cdot \ob = 0$.
\end{definition}
We note that the term \emph{orthogonal} is used as a short-hand; the vector space $F^d$ with the dot product may not be an inner product space.

If $\of= (f_0,\dots,f_{d-1})$ is a tuple of functions from $X$ to $F$ and $\oa= (a_0,\dots,a_{d-1})$ is a tuple of scalars from $F$, it is convenient to write the linear combination $a_0f_0+\dots +a_{d-1}f_{d-1}$ as the dot product $\oa \cdot \of$. 

\begin{definition}
    Let $X$ be a set, $F$ a field, $d \in \ZZp$, $\oa \in F^d$, and $\of : X \rightarrow F^d$.  Define the \emph{zero set} of $\of$ and $\oa$ by
    \[
        \ZO_{\of,\oa} = \{ c \in X : \oa \cdot \of(c) = 0 \}.
    \]
    Define
    \[
        \CC_{\of} = \left\{ \ZO_{\of,\oa} : \oa \in F^d \setminus \{ \ozero \} \right\}.
    \]
\end{definition}

\begin{example}\label{Example_Conics}
    If $X = \RR^2$, $F = \RR$, $d = 6$, and
    \[
        \of(x,y) = (x^2,xy,y^2,x,y,1),
    \]
    then $\CC_{\of}$ is the set of all conic sections in $\RR^2$.
\end{example}

\begin{lemma}\label{Lemma_LinearlyIndependent}
    Let $X$ be a set, $F$ a field, $d \in \ZZp$, and $\of : X \rightarrow F^d$.  The following are equivalent:
    \begin{enumerate}
        \item $\{ f_0, f_1, \dots, f_{d-1} \}$ is a linearly independent subset of the $F$-vector space ${}^X F$;
        \item for all $\oa \in F^d \setminus \{ \ozero \}$, there exists $c \in X$ such that $\oa \cdot \of(c) \neq 0$; and
        \item $\of(X)$ is not contained in a proper subspace of $F^d$.
    \end{enumerate}
\end{lemma}

\begin{proof}
    (1) $\Rightarrow$ (2): If $\{ f_0, f_1, \dots, f_{d-1} \}$ is a linearly independent set, then, for all $\oa \in F^d$, $\sum_{k < d} a_k f_k$ is identically zero if and only if $\oa = \ozero$.  Therefore, for any non-zero $\oa \in F^d$, $\sum_{k < d} a_k f_k$ is not identically zero, so there exists $c \in X$ such that $\oa \cdot \of(c) \neq 0$.
    
    (2) $\Rightarrow$ (3): Suppose that (3) fails.  So $\of(X)$ is contained in a proper subspace $V$ of $F^d$.  Take $\oa \in F^d$ non-zero and orthogonal to $V$.  Then, for all $c \in X$, $\oa \cdot \of(c) = 0$.  Hence, (2) fails.
    
    (3) $\Rightarrow$ (1): Suppose that (1) fails.  So, there exists $\oa \in F^d \setminus \{ \ozero \}$ such that $\sum_{k < d} a_k f_k$ is identically zero.  Therefore, $\of(X)$ is contained in the subspace of $F^d$ orthogonal to $\oa$.
\end{proof}

\begin{definition}
    Let $X$ be a set, $F$ a field, $d \in \ZZp$, and $\of : X \rightarrow F^d$.  If $\of$ satisfies any of the conditions of Lemma \ref{Lemma_LinearlyIndependent}, we say that $\of$ is \emph{linearly independent}.
\end{definition}

In Proposition \ref{Proposition_LittlestoneDim} below, we show that the Littlestone dimension of $\CC_{\of}$ is bounded above by $d-1$.  This is proven by contradiction: from a well-labeled $(X, \CC_{\of})$-labeled tree of depth $d$, we create a $d \times d$ matrix that has both full rank and less than full rank.

\begin{proposition}\label{Proposition_LittlestoneDim}
    Let $X$ be a set, $F$ a field, $d \in \ZZp$, and $\of : X \rightarrow F^d$.  Then, $\ldim(\CC_{\of}) < d$.
\end{proposition}

\begin{proof}
    For each natural number $k$, let $1_k$ denote the function from $[k]$ to $\NN$ that is constantly $1$.  Assume that $T$ is a well-labeled $(X,\CC_{\of})$-labeled tree of depth $d$.  First, we will show that the set
    \[
        V = \left\{ \of(T(1_k)) : k \in [d] \right\} \subseteq F^d
    \]
    is linearly dependent.
    
    Since $T(1_d) \in \CC_{\of}$, there exists $\oa \in F^d \setminus \{ \ozero \}$ such that $T(1_d) = \ZO_{\of,\oa}$.  On the other hand, since $T$ is well-labeled, $T(1_k) \in T(1_d)$ for all $k < d$.  Therefore, for all $k < d$,
    \[
        \oa \cdot \of(T(1_k)) = 0.
    \]
    Therefore, we have
    \[
        \left[
            \begin{array}{c c c}
                f_0(T(1_0)) & \cdots & f_{d-1}(T(1_0)) \\
                \vdots & \ddots & \vdots \\
                f_0(T(1_{d-1})) & \cdots & f_{d-1}(T(1_{d-1}))
            \end{array}
        \right]
        \left[
            \begin{array}{c}
                a_0 \\
                \vdots \\
                a_{d-1}
            \end{array}
        \right] = 
        \left[
            \begin{array}{c}
                0 \\
                \vdots \\
                0
            \end{array}
        \right].
    \]
    Since $\oa \neq \ozero$, we conclude that $V$ is linearly dependent.
    
    Since $V$ is linearly dependent, there exists $0 \le k < d$ and $\ob \in F^k$ such that
    \[
        \of(T(1_k)) = \sum_{j=0}^{k-1} b_j \of(T(1_j)).
    \]
    Consider $\tau : [d] \rightarrow [2]$ given by
    \[
        \tau(j) = \begin{cases} 1 & \text{ if } j < k \\ 0 & \text{ if } k \le j < d \end{cases}. 
    \]
    Then, for all $j < k$, $T(1_j) \in T(\tau)$ and $T(1_k) \notin T(\tau)$.  Choose $\oa \in F^d \setminus \{ \ozero \}$ such that $T(\tau) = \ZO_{\of,\oa}$.  Then, for all $j < k$,
    \[
        \oa \cdot \of(T(1_j)) = 0 \text{ and } \oa \cdot \of(T(1_k)) \neq 0.
    \]
    However,
    \begin{align*}
        \oa \cdot \of(T(1_k)) = & \ \oa \cdot \left( \sum_{j=0}^{k-1} b_j \of(T(1_j)) \right) = \\ & \ \sum_{j=0}^{k-1} b_j ( \oa \cdot \of(T(1_j)) ) = 0.
    \end{align*}
    This is a contradiction.
    
    Therefore, no such well-labeled tree $T$ exists.  That is, $\ldim(\CC_{\of}) < d$.
\end{proof}

Next, in Theorem \ref{Theorem_VCLDim} below we establish that, if $\of$ is linearly independent, then the VC-dimension and Littlestone dimension of $\CC_{\of}$ are both exactly $d-1$.  This is done through Lemma \ref{Lemma_VCDim}, where we create a new tuple of functions in the span of $\of$ that acts as an indicator function on a subset of $X$ of size $d$.  Then, in Proposition \ref{Proposition_VCDim}, we use this auxiliary tuple of functions to produce a large subset of $X$ that is shattered by $\CC_{\of}$.

\begin{lemma}\label{Lemma_VCDim}
    Let $X$ be a set, $F$ a field, $d \in \ZZp$, and $\of : X \rightarrow F^d$ linearly independent.  Then, there exists $\oc \in X^d$ and $\og : X \rightarrow F^d$ such that $g_k \in \Span \{ f_0, \dots, f_{d-1} \}$ for all $k < d$ and, for all $i, j < d$,
    \[
        g_j(c_i) = \begin{cases} 1 & \text{ if } j = i \\ 0 & \text{ otherwise } \end{cases}.
    \]
\end{lemma}

\begin{proof}
    By induction on $d$.  If $d = 1$, choose any $c_0 \in X$ where $f_0(c_0) \neq 0$ and set $g_0(x) = (f_0(c_0))^{-1} f_0(x)$.
    
    Fix $d > 1$ and suppose that $\of : X \rightarrow F^d$ is linearly independent.  By the induction hypothesis, there exists $\oc \in X^{d-1}$ and $\og' : X \rightarrow F^{d-1}$ such that $g'_k \in \Span \{ f_0, \dots, f_{d-2} \}$ for all $k < d-1$ and, for all $i, j < d - 1$,
    \[
        g'_j(c_i) = \begin{cases} 1 & \text{ if } j = i \\ 0 & \text{ otherwise } \end{cases}.
    \]
    Define
    \[
        g'_{d-1}(x) = f_{d-1}(x) - \sum_{i=0}^{d-2} f_{d-1}(c_i) g'_i(x).
    \]
    Then, $g'_{d-1}(c_i) = 0$ for all $i < d - 1$.  On the other hand, since $g'_{d-1}$ is a non-trivial linear combination of $\{ f_0, \dots, f_{d-1} \}$, which is linearly independent, $g'_{d-1}$ is non-zero.  Thus, there exists $c_{d-1} \in X$ such that $g'_{d-1}(c_{d-1}) \neq 0$.  Finally, set
    \[
        g_{d-1}(x) = (g'_{d-1}(c_{d-1}))^{-1} g'_{d-1}(x).
    \]
    and set
    \[
        g_i(x) = g'_i(x) - g'_i(c_{d-1}) g_{d-1}(x)
    \]
    for $i < d - 1$.  Check that this gives the desired conclusion.
\end{proof}

\begin{proposition}\label{Proposition_VCDim}
    Let $X$ be a set, $F$ a field, $d \in \ZZp$, and $\of : X \rightarrow F^d$ linearly independent.  Then, $\vcdim(\CC_{\of}) \ge d-1$.
\end{proposition}

\begin{proof}
    By Lemma \ref{Lemma_VCDim}, there exists $\oc \in X^d$ and $\og : X \rightarrow F^d$ such that $g_k \in \Span \{ f_0, \dots, f_{d-1} \}$ for all $k < d$ and, for all $i, j < d$,
    \[
        g_j(c_i) = \begin{cases} 1 & \text{ if } j = i \\ 0 & \text{ otherwise } \end{cases}.
    \]
    For any non-empty $S \subseteq [d]$, let
    \[
        g_S(x) = \sum_{i \in S} g_i(x).
    \]
    Then, for any non-empty $S \subseteq [d]$ and any $i < d$, note that
    \[
        g_S(c_i) = 0 \Longleftrightarrow i \notin S.
    \]
    On the other hand, since $g_k \in \Span \{ f_0, \dots, f_{d-1} \}$ and $g_k$ is not identically zero, for each non-empty $S \subseteq [d]$, there exists $\oa_S \in F^d \setminus \{ 0 \}$ such that $g_S = \oa_S \cdot \of$.  Therefore, for any non-empty $S \subseteq [d]$ and any $i < d$, 
    \[
        c_i \in \ZO_{f,a_s} \Longleftrightarrow i \notin S.
    \]
    Thus, $\CC_{\of}$ shatters $\{ c_0, \dots, c_{d-2} \}$.  So the VC-dimension of $\CC_{\of}$ is at least $d-1$.
\end{proof}

\begin{remark}
    We see from the above proof that, if $\of : X \rightarrow F^d$ is linearly independent, then $\pi_{\CC_{\of}}(d) \ge 2^d - 1$.  Using the next theorem together with the Sauer--Shelah Lemma, we conclude that $\pi_{\CC_{\of}}(d) = 2^d - 1$.
\end{remark}

\begin{theorem}\label{Theorem_VCLDim}
    Let $X$ be a set, $F$ a field, $d \in \ZZp$, and $f : X \rightarrow F^d$ linearly independent.  Then,
    \[
        \vcdim(\CC_{\of}) = \ldim(\CC_{\of}) = d-1.
    \]
\end{theorem}

\begin{proof}
    By Proposition \ref{Proposition_LittlestoneDim}, Remark \ref{Remark_VCDimLDimBound}, and Proposition \ref{Proposition_VCDim}, we obtain
    \[
        d-1 \le \vcdim(\CC_{\of}) \le \ldim(\CC_{\of}) < d.
    \]
    The conclusion follows.
\end{proof}

\section{VC-density and Littlestone density of the system of zero sets.}

Now that we have established the VC-dimension and Littlestone dimension of $\CC_{\of}$, we turn our attention to analyzing the VC-density and Littlestone density of $\CC_{\of}$.

\begin{proposition}\label{Proposition_DensityZero}
    Let $X$ be a set, $F$ a field, $d \in \ZZp$, and $\of : X \rightarrow F^d$.  Suppose that there exist $1$-dimensional subspaces $V_i \subseteq F^d$ for $i < k$ and suppose that $\of(X) \subseteq \bigcup_{i < k} V_i$.  Then, $|\CC_{\of}| \le 2^k$.  In particular, $\CC_{\of}$ has VC-density $0$ and Littlestone density $0$.
\end{proposition}

\begin{proof}
    Without loss of generality, we may assume that $V_i \neq V_j$ for all $i \neq j$.  Therefore, since $V_i$ and $V_j$ are $1$-dimensional subspaces, $V_i \cap V_j = \{ \ozero \}$.  For each $i < k$, let
    \[
        S_i = \of^{-1}(V_i \setminus \{ \ozero \})
    \]
    and let $S_k = \of^{-1}(\{ \ozero \})$.  For each $i < k$, consider $\CC_{\of} |_{S_i}$.  By assumption, for all $c \in S_i$, $\of(c) \in V_i \setminus \{ \ozero \}$.  Thus, for all $\oa \in F^d$ and $c \in S_i$, $\oa \cdot \of(c) = 0$ if and only if $\oa$ is orthogonal to $V_i$.  Since this is independent of the choice of $c \in S_i$, we see that $\ZO_{\of,\oa} \cap S_i = \emptyset$ or $\ZO_{\of,\oa} \cap S_i = S_i$.  Therefore,
    \[
        \CC_{\of} |_{S_i} = \{ \emptyset, S_i \}.
    \]
    Similarly, one can check that $\CC_{\of} |_{S_k} = \{ S_k \}$.  Therefore,
    \[
        \CC_{\of} \subseteq \left\{ \bigcup_{i \in I} S_i \cup S_k : I \subseteq [k] \right\}.
    \]
    Thus, $| \CC_{\of} | \le 2^k$.
\end{proof}

\begin{remark}
    Let $X$ be a set, $F$ a field, $d \in \ZZp$, and $\of : X \rightarrow F^d$.  If $F$ is a finite field, then $\CC_{\of}$ is finite (since $F^d \setminus \{ \ozero \}$ is finite).  Therefore, $\CC_{\of}$ has VC-density $0$ and Littlestone density $0$.
\end{remark}

In Proposition \ref{Proposition_MaxVCDim} below, we establish that, if $\of(X)$ is not contained in a finite union of proper subspaces of $F^d$, then $\CC_{\of}$ is maximal of VC-dimension $d-1$.  This is done by establishing an infinite sequence of elements from $X$ that have $d$-wise linearly independent images under $\of$.  Then, we use linear independence to select every subset of this sequence of size at most $d-1$ with some $Z_{\of,\ob}$, establishing the maximality of the VC-dimension.

\begin{proposition}\label{Proposition_MaxVCDim}
    Let $X$ be a set, $F$ a field, $d \in \NN$ with $d \ge 2$, and $\of : X \rightarrow F^d$.  If $\of(X)$ is not contained in a finite union of proper subspaces of $F^d$, then $\CC_{\of}$ is maximal of VC-dimension $d-1$.  In particular, $\vcden(\CC_{\of}) = d-1$.
\end{proposition}

\begin{proof}
    First, create a sequence $( c_i )_{i \in \NN}$ with $c_i \in X$ for all $i \in \NN$ such that, for all $I \in \binom{\NN}{\le d}$,
    \begin{equation}\label{Equation_LinearlyIndep}
        \left\{ \of(c_i) : i \in I \right\} \text{ is linearly independent.}
    \end{equation}
    We do this recursively as follows:
    
    First, choose $c_0 \in X$ such that $\of(c_0) \neq \ozero$ (which can be done, since $\of(X) \not\subseteq \{ \ozero \}$).  Now, assume that $c_0, \dots, c_{n-1}$ have been constructed so that, for all $I \in \binom{[n]}{\le d}$, \eqref{Equation_LinearlyIndep} holds.  Consider the set
    \[
        U = \bigcup_{I \in \binom{[n]}{< d}} \Span \left\{ \of(c_i) : i \in I \right\}.
    \]
    This is a union of finitely many proper subspaces of $F^d$, so, by the assumption, there exists $c_n \in X$ such that $\of(c_n) \notin U$.  It is easy to check that, for all $I \in \binom{[n+1]}{\le d}$, \eqref{Equation_LinearlyIndep} holds.  This concludes our construction.
    
    For all $I \in \binom{[n]}{d-1}$, consider the homogeneous system
    \[
        \left[
            \begin{array}{c c c}
                f_0(c_{i_0}) & \cdots & f_{d-1}(c_{i_0}) \\
                \vdots & \ddots & \vdots \\
                f_0(c_{i_{d-2}}) & \cdots & f_{d-1}(c_{i_{d-2}})
            \end{array}
        \right]
        \left[
            \begin{array}{c}
                x_0 \\
                \vdots \\
                x_{d-1}
            \end{array}
        \right] = 
        \left[
            \begin{array}{c}
                0 \\
                \vdots \\
                0
            \end{array}
        \right],
    \]
    where $i_0 < \dots < i_{d-2}$ enumerate $I$.  Since the matrix has rank $d-1$, this has a non-trivial solution $\ob_I \in F^d$.  That is, for all $i \in I$, $\ob_I \cdot \of(c_i) = 0$.  On the other hand, for any $j \in \NN \setminus I$, $\{ \of(c_i) : i \in I \cup \{ j \} \}$ is linearly independent, so we have that
    \[
        \left[
            \begin{array}{c c c}
                f_0(c_{i_0}) & \cdots & f_{d-1}(c_{i_0}) \\
                \vdots & \ddots & \vdots \\
                f_0(c_{i_{d-2}}) & \cdots & f_{d-1}(c_{i_{d-2}}) \\
                f_0(c_j) & \cdots & f_{d-1}(c_j)
            \end{array}
        \right]
        \ob_I \neq 
        \left[
            \begin{array}{c}
                0 \\
                \vdots \\
                0 \\
                0
            \end{array}
        \right].
    \]
    (This holds since $\ob_I \neq 0$ and this matrix has rank $d$.)  Therefore, we see that $\ob_I \cdot f(c_j) \neq 0$.  In other words, we have constructed, for each $I \in \binom{\NN}{d-1}$, $\ob_I \in F^d$ such that, for all $i \in \NN$,
    \[
        \ob_I \cdot \of(c_i) = 0 \Longleftrightarrow i \in I.
    \]
    In other words,
    \[
        c_i \in \ZO_{\of,\ob_I} \Longleftrightarrow i \in I.
    \]
    For any $n \ge d$, for any $I \in \binom{[n]}{< d}$, set $I' = I \cup \{ n, n+1, \dots, n+(d-|I|-2) \}$ and we have, for all $i \in [n]$,
    \[
        c_i \in \ZO_{\of,\ob_{I'}} \Longleftrightarrow i \in I.
    \]
    Thus, 
    \[
        \left| \CC_{\of} |_{ \{ c_i : i \in [n] \} } \right| \ge \binom{n}{< d}.
    \]
    Thus, since $\vcdim(\CC_{\of}) = d-1$, $\CC_{\of}$ is maximal of VC-dimension $d-1$.
\end{proof}

Unfortunately, we do not necessarily get an upper bound on the VC-density of $\CC_{\of}$, even if $\of(X)$ is contained in a union of proper subspaces of $F^d$.

\begin{example}\label{Example_HighVCDen}
    Fix $d \in \NN$ with $d \ge 3$ and $F$ an infinite field.  For each $i < d$, let $\oee_i \in F^d$ denote the $i$th standard basis vector.  For each $i < d-1$, let
    \[
        L_i = \Span \left\{ \oee_0, \oee_{i+1} \right\}.
    \]
    Let $X = \bigcup_{i=0}^{d-2} L_i$ and let $\of : X \rightarrow F^d$ be the identity embedding.  Then, $\vcden(\CC_{\of}) = d-1$, even though $\of(X)$ is contained in the union of $2$-dimensional subspaces of $F^d$.
\end{example}

\begin{proof}
    Since $F$ is infinite, there exists $g : \NN \rightarrow F \setminus \{ 0 \}$ an injective function.  For each $i < d-1$ and $j \in \NN$, define
    \[
        \oc_{i,j} = \oee_0 + g(j) \oee_{i+1}.
    \]
    For each $j_0, \dots, j_{d-2} \in \NN$, let
    \[
        \ob_{j_0, \dots, j_{d-2}} = \left( \prod_{k<d-1} g(j_k) \right) \oee_0 + \sum_{\ell<d-1} \left( - \prod_{k < d-1, k \neq \ell} g(j_k) \right) \oee_{\ell+1}.
    \]
    Then, for all $i < d-1$, $j \in \NN$, and $j_0, \dots, j_{d-2} \in \NN$,
    \begin{align*}
        \ob_{j_0, \dots, j_{d-2}} \cdot \of(c_{i,j}) = & \ \ob_{j_0, \dots, j_{d-2}} \cdot \oc_{i,j} = \\ & \ \prod_{k<d-1} g(j_k) - g(j) \prod_{k < d-1, k \neq i} g(j_k) = \\ & \ (g(j_i) - g(j)) \prod_{k < d-1, k \neq i} g(j_k).
    \end{align*}
    Therefore, $\oc_{i,j} \in \ZO_{\of,\ob_{j_0,\dots,j_{d-2}}}$ if and only if $g(j) = g(j_i)$ if and only if $j = j_i$.  Thus, for any $n \in \NN$,
    \[
        \left| \CC_{\of} |_{ \{ \oc_{i,j} : i < d-1, j < n \} } \right| \ge n^{d-1}.
    \]
    Thus, $\vcden(\CC_{\of}) \ge d-1$.  However, by Theorem \ref{Theorem_VCLDim} and Remark \ref{Remark_VCDenVCDim}, we conclude that $\vcden(\CC_{\of}) = d-1$.
\end{proof}

\begin{corollary}\label{Corollary_MaxLDim}
    Let $X$ be a set, $F$ a field, $d \in \NN$ with $d \ge 2$, and $\of : X \rightarrow F^d$.  If $\of(X)$ is not contained in a finite union of proper subspaces of $F^d$, then $\CC_{\of}$ is maximal of Littlestone dimension $d-1$.  In particular, $\lden(\CC_{\of}) = d-1$.
\end{corollary}

\begin{proof}
    By Proposition \ref{Proposition_MaxVCDim}, $\CC_{\of}$ is maximal of VC dimension $d-1$.  By Theorem \ref{Theorem_VCLDim}, $\CC_{\of}$ has Littlestone dimension $d-1$.  By Remark \ref{Remark_VCDenVCDim}, $\CC_{\of}$ is maximal of Littlestone dimension $d-1$.
\end{proof}

\section{Characterization of maximality of systems of zero sets.}  

In this section, we prove the converse to Proposition~\ref{Proposition_MaxVCDim}.

\begin{definition}
    Let $F$ be a field, $d \in \ZZp$, $S \subseteq F^d$ finite, and $\CC \subset \Pow(S)$.  We say that $\CC$ is \emph{span injective} if 
    \begin{enumerate}
        \item for all $A \in \CC$, $\Span(A) \neq F^d$, and
        \item for all $A, B \in \CC$, if $\Span(A) = \Span(B)$, then $A = B$.
    \end{enumerate}
    In other words, $\CC$ is span injective if $\Span$ is an injective function from $\CC$ to the set of all proper subsets of $F^d$.  
\end{definition}

\begin{lemma}\label{Lemma_CfSpanInj}
    Let $X$ be a set, $F$ a field, $d \in \ZZp$, and $\of : X \rightarrow F^d$.  Then, the set
    \[
        \left\{ \of(Z_{\of,\oa}) : \oa \in F^d \setminus \{ \ozero \} \right\}
    \]
    is span injective.
\end{lemma}

\begin{proof}
    For all $\oa \in F^d \setminus \{ \ozero \}$, for all $c \in Z_{\of,\oa}$, $\oa \cdot \of(c) = 0$, so every element of $\of(Z_{\of,\oa})$ is orthogonal to $\oa$.  Hence, $\Span( \of(Z_{\of,\oa}) ) \neq F^d$.  Suppose that $\oa, \ob \in F^d \setminus \{ \ozero \}$ and $\Span(\of(Z_{\of,\oa})) = \Span(\of(Z_{\of,\ob}))$.  Then, as $\oa$ is orthogonal to $\of(Z_{\of,\oa})$, it is orthogonal to $\Span(\of(Z_{\of,\oa}))$, hence it is orthogonal to $\Span(\of(Z_{\of,\ob}))$, so $\oa$ is orthogonal to $\of(Z_{\of,\ob})$.  Thus, for any $c \in Z_{\of,\ob}$, $\oa \cdot \of(c) = 0$, so $c \in Z_{\of,\oa}$.  By symmetry, we obtain that $Z_{\of,\oa} = Z_{\of,\ob}$.
\end{proof}

\begin{lemma}\label{Lemma_SpanInjSmall}
    Let $F$ be a field, $d \in \ZZp$, $S \subseteq F^d$ finite, and $\CC \subseteq \Pow(S)$.  If $\CC$ is span injective, then there exists $\CC' \subseteq \Pow(S)$ such that
    \begin{enumerate}
        \item $|\CC'| = |\CC|$,
        \item $\CC'$ is span injective,
        \item the elements of $\CC'$ are linearly independent, and
        \item $\CC' \subseteq \binom{S}{< d}$.
    \end{enumerate}
\end{lemma}

\begin{proof}
    For each $A \in \CC$, let $I_A \subseteq A$ be minimal such that $\Span(I_A) = \Span(A)$ and let
    \[
        \CC' = \{ I_A : A \in \CC \}.
    \]
    Clearly $\CC'$ is span injective.  Fix $A \in \CC$.  Since $I_A$ is chosen to be minimal, $I_A$ is linearly independent.  Since $\CC$ is span injective, $\Span(I_A) = \Span(A) \neq F^d$, so $|I_A| < d$.  Thus,
    \[
        \CC' \subseteq \binom{S}{< d}.
    \]
    Fix $A, B \in \CC$ and suppose that $I_A = I_B$.  Then,
    \[
        \Span(A) = \Span(I_A) = \Span(I_B) = \Span(B).
    \]
    Since $\CC$ is span injective, $A = B$.  Therefore, $A \mapsto I_A$ is a bijection from $\CC$ to $\CC'$, so $|\CC| = |\CC'|$.
\end{proof}

\begin{lemma}\label{Lemma_NotMaximal}
    Let $F$ be a field, $d \in \ZZp$, $S \subseteq F^d$ finite, and $\CC \in \Pow(S)$.  If $|S| > k(d-1)$, $S$ is contained in the union of $k$ proper subsets of $F^d$, and $\CC$ is span injective, then
    \[
        |\CC| < \binom{|S|}{< d}.
    \]
\end{lemma}

\begin{proof}
    Let $\CC'$ be given as in Lemma \ref{Lemma_SpanInjSmall}.  Since $|\CC| = |\CC'|$ and $\CC' \subseteq \binom{S}{< d}$, it suffices to show that there exists a $(d-1)$-element subset of $S$ not in $\CC'$.
    
    Since $S$ is contained in the union of $k$ proper subsets of $F^d$ and $|S| > k(d-1)$, by the pigeonhole principle, there exists a proper subspace $L$ of $F^d$ that contains at least $d$ elements of $S$.  Suppose that there exists $A, B \in \binom{S \cap L}{d-1}$ such that $A, B \in \CC'$.  Since $A$ and $B$ are linearly independent of cardinality $d-1$ and $L$ has dimension less than $d$, $\Span(A) = \Span(B) = L$.  Since $\CC'$ is span injective, $A = B$.  That is, only one element of $\binom{S \cap L}{d-1}$ belongs to $\CC'$.
    
    Therefore,
    \[
        |\CC| = |\CC'| < \binom{|S|}{<d}.
    \]
\end{proof}

We are ready to establish the converse to Proposition~\ref{Proposition_MaxVCDim}.

\begin{proposition}\label{Proposition_NotMaxVCDim}
    Let $X$ be a set, $F$ a field, $d \in \NN$ with $d \ge 2$, and $\of : X \rightarrow F^d$.  If $\of(X)$ is contained in a finite union of proper subspaces of $F^d$, then $\CC_{\of}$ is not maximal of VC-dimension $d-1$.
\end{proposition}

\begin{proof}
    Let $n = k(d-1)+1$ and fix some $X_0 \in \binom{X}{n}$.  Consider $\of |_{X_0} : X_0 \rightarrow F^d$ and, for each $\oa \in F^d \setminus \{ \ozero \}$, the corresponding $Z_{\of |_{X_0}, \oa} \in \Pow(X_0)$.  By Lemma \ref{Lemma_CfSpanInj}, the set
    \[
        \CC = \{ \of |_{X_0} (Z_{\of |_{X_0},\oa}) : \oa \in F^d \setminus \{ \ozero \} \}
    \]
    is span injective.  By Lemma \ref{Lemma_NotMaximal},
    \[
        | \CC | < \binom{n}{<d}.
    \]
    It is clear that, for all $\oa \in F^d \setminus \{ \ozero \}$,
    \[
        Z_{\of,\oa} \cap X_0 = Z_{\of |_{X_0},\oa}.
    \]
    Thus, $\CC_{\of} |_{X_0} = \CC_{\of |_{X_0}}$.  Therefore,
    \[
        | \CC_{\of} |_{X_0} | = | \CC_{\of |_{X_0}} | \le | \CC | < \binom{n}{<d}.
    \]
    Thus, $\CC_{\of}$ is not maximal of VC-dimension $d-1$.
\end{proof}

We thus obtain the following result.

\begin{theorem}\label{Theorem_MaxVCDim}
    Let $X$ be a set, $F$ a field, $d \in \NN$ with $d \ge 2$, and $\of : X \rightarrow F^d$.  Then, the following are equivalent:
    \begin{enumerate}
        \item $\CC_{\of}$ is maximal of VC-dimension $d-1$;
        \item $\of(X)$ is not contained in a finite union of proper subspaces of $F^d$.
    \end{enumerate}
\end{theorem}

\begin{proof}
    (1) $\Rightarrow$ (2) follows from the contrapositive of Proposition \ref{Proposition_NotMaxVCDim} and (2) $\Rightarrow$ (1) follows from Proposition \ref{Proposition_MaxVCDim}.
\end{proof}

It turns out that the conditions in Theorem \ref{Theorem_MaxVCDim} are not equivalent to $\CC_{\of}$ being maximal \emph{Littlestone} dimension $d-1$.

\begin{example}
    Consider the function $\of : X \rightarrow F^d$ in Example \ref{Example_HighVCDen}.  That is, $F$ is an infinite field, $d \ge 3$,
    \[
        X = \bigcup_{i=0}^{d-2} \Span \{ \oee_0, \oee_{i+1} \},
    \]
    and $\of: X \rightarrow F^d$ is the identity embedding.  Then, $\CC_{\of}$ is maximal of Littlestone dimension $d-1$, even though $\of(X)$ is containd in a union of $2$-dimensional subspaces of $F^d$.
\end{example}

\begin{proof}
    Define $\oc_{i,j}$ for $i \in [d-1]$ and $j \in \NN$ and define $\ob_{j_0, \dots, j_{d-2}}$ for $j_0, \dots, j_{d-2} \in \NN$ as in Example \ref{Example_HighVCDen}.  In particular, $\oc_{i,j} \in X$ for all $i \in [d-1]$ and $j \in \NN$ and
    \[
        \oc_{i,j} \in Z_{\of, \ob_{j_0,\dots,j_{d-2}}} \Longleftrightarrow j = j_i
    \]
    for all $i \in [d-1]$, $j, j_0, \dots, j_{d-2} \in \NN$.  Fix $n \in \NN$ and we define a $(X, \CC_{\of})$-labeled tree $T$ of depth $n$ with $\binom{n}{<d}$ well-labeled leaves, showing that $\CC_{\of}$ is maximal of Littlestone dimension $d-1$.
    
    For each $k \in [n]$ and $\sigma \in {}^{[k]} [2]$, let $i = | \supp(\sigma) |$ and let
    \[
        T(\sigma) = \begin{cases} \oc_{i,k} & \text{ if } i < d-1 \\ \oc_{0,n} & \text{ otherwise} \end{cases}
    \]
    For $\tau \in {}^{[n]} [2]$, let $j_0 < j_1 < \dots < j_{\ell-1}$ enumerate the set $\supp(\sigma)$.  If $\ell < d-1$, let $j_\ell = j_{\ell+1} = \dots = j_{d-2} = n$.  Then, let
    \[
        T(\tau) = Z_{\of, \ob_{j_0,\dots,j_{d-2}}}.
    \]
    We claim that, for each $\tau \in {}^{[n]}[2]$, if $|\supp(\tau)| < d$, then $(\tau, T(\tau))$ is well-labeled.  This exhibits $\binom{n}{<d}$ well-labeled leaves of $T$.
    
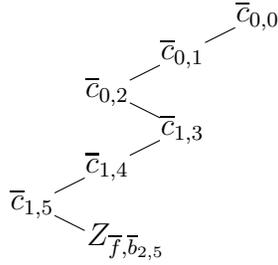
    \begin{center}
        \begin{tikzpicture}
            \draw (0,3) node {$\oc_{0,0}$};
            \draw (-0.3,2.85) -- (-0.7,2.65);
            \draw (-1,2.5) node {$\oc_{0,1}$};
            \draw (-1.3,2.35) -- (-1.7,2.15);
            \draw (-2,2) node {$\oc_{0,2}$};
            \draw (-1.7,1.85) -- (-1.3,1.65);
            \draw (-1,1.5) node {$\oc_{1,3}$};
            \draw (-1.3,1.35) -- (-1.7,1.15);
            \draw (-2,1) node {$\oc_{1,4}$};
            \draw (-2.3,0.85) -- (-2.7,0.65);
            \draw (-3,0.5) node {$\oc_{1,5}$};
            \draw (-2.7,0.35) -- (-2.3,0.15);
            \draw (-1.75,0) node {$Z_{\of, \ob_{2,5}}$};
        \end{tikzpicture}
    \end{center}
    
    Fix $\tau \in {}^{[n]} [2]$ with $|\supp(\tau)| < d$.  Say $j_0 < \dots < j_{\ell-1}$ enumerates $\supp(\tau)$.  Clearly $\ell \le d-1$.  Set $j_\ell = j_{\ell+1} = \dots = j_{d-2} = n$ and $j_{-1} = -1$.  Fix $k \in [n]$.  If $j_{i-1} < k \le j_i$ for some $i \in [d-1]$, then
    \[
        |\supp(\tau |_{[k]})| = | \{ j_0, ..., j_{i-1} \} | = i.
    \]
    Therefore, $T(\tau |_{[k]}) = \oc_{i,k}$.  Moreover,
    \[
        T(\tau) = Z_{\of, \ob_{j_0,\dots,j_{d-2}}}.
    \]
    By construction, $T(\tau |_{[k]}) \in T(\tau)$ if and only if $k = j_i$ if and only if $\tau(k) = 1$.  Finally, if $j_{d-2} < k < n$, then $| \supp(\tau |_{[k]}) | = d-1$, so $T(\tau |_{[k]}) = \oc_{0,n}$.  Thus, since $j_0 < n$, $T(\tau |_{[k]}) \notin T(\tau)$.  Moreover, $\tau(k) = 0$.
    
    In any case, we see that, for all $k \in [n]$, $T(\tau |_{[k]}) \in T(\tau)$ if and only if $\tau(k) = 1$.  Therefore, we conclude that $(\tau, T(\tau))$ is well-labeled.
\end{proof}

\section{Conclusion}

We have determined the VC-dimension and Littlestone dimension of $\CC_{\of}$, as well as given a characterization for when $\CC_{\of}$ is maximal in VC-dimension.  We can apply these results, for example, to the set of conic sections in $\RR^2$ via Example \ref{Example_Conics}.

\begin{example}
    Let $X = \RR^2$, $F = \RR$, $d = 6$, and
    \[
        \of(x,y) = (x^2,xy,y^2,x,y,1).
    \]
    As noted above, $\CC_{\of}$ is the set of all conic sections in $\RR^2$.  For example, we have in the image of $\of$,
    \begin{align*}
        \of(0,0) = & \ (0, 0, 0, 0, 0, 1), \\
        \of(1,0) = & \ (1, 0, 0, 1, 0, 1), \\
        \of(0,1) = & \ (0, 0, 1, 0, 1, 1), \\
        \of(1,1) = & \ (1, 1, 1, 1, 1, 1), \\
        \of(2,1) = & \ (4, 2, 1, 2, 1, 1), \\
        \of(1,2) = & \ (1, 2, 4, 1, 2, 1).
    \end{align*}
    It is easy to check that these vectors form a basis for $\RR^6$.  By Lemma \ref{Lemma_LinearlyIndependent}, $\of$ is linearly independent.  Thus, by Theorem \ref{Theorem_VCLDim}, $\CC_{\of}$ has VC-dimension and Littlestone dimension $5$.  Moreover, $\of(\RR^2)$ is not contained in a finite union of proper subspaces of $\RR^6$.
    
    Towards a contradiction, suppose that $\of(\RR^2)$ is contained in finitely many proper subspaces of $\RR^6$, say $L_0, \dots, L_{m-1}$.  For each $i < m$, suppose that $\oa_i \in \RR^6$ is orthogonal to $L_i$.  For each $\ob \in \RR^2$, $\of(\ob) \in \of(\RR^2)$, hence there exists $i < m$ such that $\of(\ob) \in L_i$.  That is, $\oa_i \cdot \of(\ob) = \ozero$, so $\ob \in Z_{\of,\oa_i}$.  Thus, 
    \[
        \RR^2 \subseteq \bigcup_{i < m} Z_{\of,\oa_i}.
    \]
    However, each $Z_{\of,\oa_i}$ is a conic section.  This is a contradiction, since we cannot cover $\RR^2$ with finitely many conic sections.
    
    Therefore, by Theorem \ref{Theorem_MaxVCDim}, $\CC_{\of}$ is maximal of VC-dimension $5$.  By Corollary \ref{Corollary_MaxLDim}, $\CC_{\of}$ is maximal of Littlestone dimension $5$.  In particular, the set of conic sections in $\RR^2$ has VC-density $5$ and Littlestone density $5$.
\end{example}

The results of this paper can be applied to more general situations.  For example, consider the set of axes-aligned ellipses.  We can use Theorem \ref{Theorem_VCLDim} to compute the VC-dimension and Littlestone dimension of this class.

\begin{example}
    Let $\CC$ be the set of axes-aligned ellipses in $\RR^2$.  Formally, let
    \[
        E_{a,b,c,d} = \left\{ (x,y) \in \RR^2 : a(x-b)^2 + c(y-d)^2 = 1 \right\}
    \]
    and let $\CC$ be the class of all $E_{a,b,c,d}$ where $a,b,c,d \in \RR$ and $a,c > 0$.  Note that $\CC$ is a subclass of the class $\CC_{\of}$, where $\of : \RR^2 \rightarrow \RR^5$ given by
    \[
        \of(x,y) = (x^2, y^2, x, y, 1).
    \]
    It is easy to show that $\of$ is linearly independent.  By Theorem \ref{Theorem_VCLDim}, $\CC_{\of}$ has VC-dimension and Littlestone dimension $4$.  Moreover, it is not hard to show that $\CC$ has VC-dimension at least $4$ (for example, the set $\{ (1,0), (0,1), (-1,0), (0,-1) \}$ is shattered by $\CC$).  Therefore,
    \[
        4 \le \vcdim(\CC) \le \ldim(\CC) \le \ldim(\CC_{\of}) = 4.
    \]
    Thus, $\CC$ has VC-dimension and Littlestone dimension $4$.  On the other hand, although we can show that $\CC_{\of}$ is maximal of VC-dimension and Littlestone dimension $4$, it is more difficult to apply Theorem \ref{Theorem_MaxVCDim} on the class of axes-aligned ellipses.  
\end{example}

In general, it may be interesting to examine what happens when we restrict the parameters of $\CC_{\of}$.  Under what conditions does the VC-dimension or Littlestone dimension drop below $d-1$?  What can be said about the maximality of VC-dimension or Littlestone dimension in this case?

\end{document}